\documentclass{amsart}
\usepackage{amsthm,amsfonts,amsmath,amscd,amssymb,latexsym,epsfig}
\newcommand{\gl}{{\mathfrak g \mathfrak l}}

\newcommand{\su}{{\mathfrak s  \mathfrak u}}

\newcommand{\cx}{{\mathbb C}}

\newcommand{\tr}{\operatorname{tr}}

\numberwithin{equation}{section}

\newtheorem{theorem}{Theorem}[section]

\newtheorem{corollary}[theorem]{Corollary}

\newtheorem{proposition}[theorem]{Proposition}

\theoremstyle{remark}

\newtheorem{remark}[theorem]{Remark}

\newtheorem{definition}[theorem]{Definition}

\newtheorem{example}[theorem]{Example}

\newcommand{\oC}{{\mathbb{C}}}

\newcommand{\oH}{{\mathbb{H}}}

\newcommand{\oP}{{\mathbb{P}}}

\newcommand{\oR}{{\mathbb{R}}}

\newcommand{\sF}{{\mathcal{F}}}

\newcommand{\sM}{{\mathcal{M}}}   
\newcommand{\sN}{{\mathcal{N}}}
\newcommand{\sO}{{\mathcal{O}}}

\newcommand{\fU}{{\mathfrak{u}}}

\begin{document}

\title[Transverse Hilbert schemes and bi-Hamiltonian systems]{Transverse Hilbert schemes, bi-Hamiltonian systems, and hyperk\"ahler geometry}
\author{Roger Bielawski}
\dedicatory{Dedicated to the memory of Sir Michael Francis Atiyah (1929-2019)}
\address{Institut f\"ur Differentialgeometrie,
Leibniz Universit\"at Hannover,
Welfengarten 1, 30167 Hannover, Germany}

\thanks{The author  is a member of the DFG Priority
Programme 2026 ``Geometry at infinity".}


\begin{abstract} We  give a characterisation of Atiyah's and Hitchin's transverse Hilbert schemes of points on a symplectic surface in terms of bi-Poisson structures. Furthermore, we describe the geometry of hyperk\"ahler manifolds arising from the transverse Hilbert scheme construction, with particular attention paid to the monopole moduli spaces.
\end{abstract}

\maketitle

\thispagestyle{empty}

\section{Introduction}

In chapter 6 of the monograph \cite{AH}, Atiyah and Hitchin consider the following construction. Let $Y$ be a complex symplectic surface with a  holomorphic submersion $\pi$ onto a $1$-dimensional complex manifold $X$. They associate to it an open subset of the Hilbert scheme of $n$ points on $Y$ consisting of  $0$-dimensional complex subspaces $D$ of length $n$ such that $\pi_{|D}$ is an isomorphism onto its scheme-theoretic image. They observe that this {\em transverse Hilbert scheme} $Y^{[n]}_\pi$ is a symplectic manifold equipped with holomorphic submersion $\pi^{[n]}$ onto $S^n X$, the fibres of which are Lagrangian submanifolds. In particular, if $X$ is a domain in $\cx$, then  the components of $\pi^{[n]}$ define $n$ functionally independent and Poisson-commuting Hamiltonians on $Y^{[n]}_\pi$, i.e.\ a completely integrable system. Atiyah and Hitchin observe further that sometimes one can perform this construction on the fibres of the twistor space of a $4$-dimensional hyperk\"ahler manifold and obtain 
a new twistor space which then might lead to a new hyperk\"ahler manifold. Their main example of this construction is $Y=\cx^\ast \times \cx$ with $\pi$ the projection onto the second factor. The corresponding transverse Hilbert scheme is the space of based rational maps of degree $n$ and the hyperk\"ahler metric resulting from applying the construction to the twistor space of $S^1\times \oR^3$ is the $L^2$-metric on the moduli space of Euclidean monopoles of charge $n$. Further examples of this construction are given in \cite{Man,BiCM}.
\par
The purpose of this article is to characterise both symplectic and  hyperk\"ahler manifolds arising from this construction. Partial results in this direction have been obtained in \cite{limit} and in \cite{Nicc1, Nicc2}.  They rely on the existence of a certain endomorphism of the tangent bundle of $Y^{[n]}_\pi$. In the present work our point of view is different. We observe that $Y^{[n]}_\pi$ is equipped with a second Poisson structure, compatible with the symplectic form. Thus $Y^{[n]}_\pi$ is a completely integrable {\em bi-Hamiltonian system}. We then show that a nondegenerate bi-Poisson manifold $M^{2n}$ arises as an (open subset) of a transverse Hilbert scheme on a symplectic surface $Y$ with a submersion $Y\to\cx$ essentially exactly then, when the coefficients of the minimal polynomial of the corresponding recursion operation (see \S\ref{bi} for a definition) form a submersion to $\cx^n$. 
\par
We then turn our attention to hyperk\"ahler manifolds arising from the transverse Hilbert scheme construction on the fibres of the twistor space a $4$-dimensional hyperk\"ahler manifold with a tri-Hamiltonian vector field. We show that the essential feature of the geometry of a manifold $M$ arising from this construction is the existence of a bivector $\Pi$ on $M$ which lies in Salamon's component $\Lambda^2E\otimes S^2 H$ of $\Lambda^2T^\cx M$ and satisfies ${\rm D}\Pi=0$, where ${\rm D}$ is the Penrose-Ward-Salamon differential operator on $\Lambda^2E\otimes S^2 H$ \cite{Sal}.
The bivector $\Pi$ is not Poisson, but  its $(2,0)$-component with respect to each complex structure is a (generically log-symplectic) holomorphic Poisson bivector. Moreover, this holomorphic Poisson bivector is compatible with the parallel holomorphic symplectic form arising from the hyperk\"ahler structure.
\par
In the last section we identify the bivector $\Pi$ on moduli spaces of $SU(2)$-monopoles, i.e.\ hyperk\"ahler transverse Hilbert schemes on $S^1\times \oR^3$, in terms of solutions to Nahm's equations.

\section{Transverse Hilbert schemes and bi-Hamiltonian systems}

\subsection{Transverse Hilbert schemes\label{trH}}

Let $X$ be a complex manifold, $C$ a complex manifold of dimension $1$, and $\pi:X\to C$ a holomorphic map. The {\em transverse Hilbert scheme} $X^{[n]}_\pi$ of $n$ points in $X$ is an open subset of the full Hilbert scheme $X^{[n]}$ consisting of those $D\in X^{[n]}$ such that $\pi_{|D}$ is an isomorphism onto its scheme-theoretic image \cite{AH}. Since $C^{[n]}=S^n C$, this simply means that $\pi(D)$ consists of $n$ points (with multiplicities).
First of all, observe that $X^{[n]}_\pi$ is always smooth, unlike the full Hilbert scheme $X^{[n]}$:
\begin{proposition} Let $\pi:X\to C$ be a holomorphic map from a complex manifold $X$ to a $1$-dimensional complex manifold $C$. Then the transverse Hilbert scheme $X^{[n]}_\pi$ is smooth.
\end{proposition}
\begin{proof} Since $D\in X^{[n]}_\pi$ satisfies $D\simeq \pi(D)\in S^nC$, such a $D$ is a local complete intersection (l.c.i.). Now the claim follows from general results of deformations theory (see, e.g., \cite[Theorem 1.1.(c)]{Hart}). 
\end{proof}
The transverse Hilbert scheme comes equipped with a canonical map $\pi^{[n]}:X^{[n]}_\pi\to S^nC$. If $\pi$ is a submersion, then so is $\pi^{[n]}$. In this case, 
points of $X^{[n]}_\pi$ such that $\pi(D)=n_1p_1+\dots+ n_kp_k$, with $p_1,\dots,p_k$ distinct points of $C$, correspond to a choice of a section $s_i$ of $\pi$ in a neighbourhood of each $p_i$, truncated to order $n_i$ (in other words $s_i$ is an $(n_i-1)$-jet of sections at $p_i$). Let us remark that Atiyah and Hitchin consider only the case when $\pi$ is a submersion (they say $\pi$ is a ``complex fibration", and the proof of the smoothness of $X^{[n]}_\pi$, given on p.\ 53 in \cite{AH}, makes clear that $\pi$ must be a submersion).

Suppose now that $X$ has a symplectic structure. If $\dim X=2$, then a theorem of Beauville \cite{Beau} implies that $X^{[n]}$ (which is smooth, owing to a well-known result of Fogarty), and hence $ X^{[n]}_\pi$, carries an induced symplectic structure. For higher dimensional $X$,  there is no induced symplectic structure on $X^{[n]}$, not even on its smooth locus.

\subsection{Log-symplectic Poisson structures}

A Poisson structure on a (smooth or complex) manifold $M^{2n}$ is given by a bivector $\Pi\in \Gamma(\Lambda^2 TM)$ such that the Schouten bracket $[\Pi,\Pi]$ vanishes. The {\em symplectic locus} of the Poisson structure is the set of points $m$ where the induced map $\#_\Pi:T_m^\ast M\to T_mM$ is an isomorphism. Its complement is called the {\em degeneracy locus}. A Poisson structure is called {\em log-symplectic} if $\Pi^n\in \Gamma(\Lambda^{2n }TM)$ meets the zero section of the $\Lambda^{2n }TM$ transversely. These structures were studied by Goto \cite{Goto} in the holomorphic case, and by  Guillemin, Miranda and Pires in the smooth category \cite{GMP} (see also \cite{Gua,Cav}). The name is justified by the fact that the dual $2$-form $\omega=\Pi^{-1}$ has a logarithmic singularity along the degeneracy locus. The degeneracy locus $\Delta$ of a log-symplectic Poisson structure is a smooth Poisson hypersurface with codimension one symplectic leaves and $M\backslash \Delta$ is a union of open symplectic leaves.
\par
We recall from \cite{GMP} that if $f$ is a local defining function for $\Delta$, then $\omega$ can be decomposed
as
\begin{equation}\omega=\alpha\wedge \frac{df}{f}+\beta,\label{normal}
\end{equation}
for a $1$-form $\alpha$ and a $2$-form $\beta$. Moreover, the restrictions of $\alpha$ and $\beta$ to $\Delta$ are closed, $\alpha_{|\Delta}$ is intrinsically defined and its kernel is the tangent space to the symplectic leaf of $\Pi$.

\subsection{Bi-Poisson structures \label{bi}}

A {\em bi-Poisson} structure on a (real or complex) manifold $M$ is a pair $(\Pi_1,\Pi_2)$ of linearly independent bivectors such that every linear combination of $\Pi_1$ and $\Pi_2$ is a Poisson structure. In other words $\Pi_1$ and $\Pi_2$ satisfy $[\Pi_1,\Pi_1]=0$,  $[\Pi_2,\Pi_2]=0$, $[\Pi_1,\Pi_2]=0$, where $[\ ,\ ]$ is the Schouten bracket.
\par
A bi-Poisson structure is called {\em nondegenerate}, if the pencil $t_1\Pi_1+t_2\Pi_2$ contains a symplectic structure. 
In what follows, we shall consider only nondegerate bi-Poisson structures and assume that $\Pi_1$ is symplectic. Following Magri and Morosi \cite{MM} (see also \cite{Rui}) we can define the {\em recursion operator} $R=\#_{\Pi_2}\circ \#_{\Pi_1}^{-1}$. It is an endomorphism of $TM$ and Magri and Morosi  show that 1) its Nijenhuis tensor vanishes; and 2) the eigenvalues of $R$ form a commuting family with respect to both  Poisson brackets.
\par
Furthermore, $\det R=(\mu_R)^2$ for a well defined function $\mu_R$ on $M$ ($\mu_R$ is the quotient of the Pfaffians of $\Pi_2$ and  of $\Pi_1$). Thus $\Pi_2$ is log-symplectic if and only if $0$ is a regular value of $\mu_R$. Since $(\Pi_1, \Pi_2-\lambda\Pi_1)$ is a nondegenerate bi-Poisson structure for each scalar $\lambda$, the characteristic polynomial of $R$ is of the form $\chi_R(\lambda)=\mu_R(\lambda)^2$. We shall refer to $\mu_R(\lambda)$ as the {\em Pfaffian polynomial} of R. We observe:
\begin{proposition} Let $(\Pi_1,\Pi_2)$ be a real (resp.\ holomorphic) bi-Poisson structure on a smooth (resp.\ complex) manifold $M^{2n}$ with $\Pi_1$ symplectic. If the coefficients of the Pfaffian polynomial of the recursion operator $R$ define a submersion $p:M\to \oR^{n}$ (resp.\ $p:M\to \oC^{n}$), then the Poisson structure $\Pi_2-\lambda\Pi_1$ is log-symplectic for every $\lambda$.\label{log}\hfill $\Box$
\end{proposition}
We also recall the following property of bi-Poisson structures, proved by Magri and Morosi in \cite{MM}:
\begin{proposition}[Magri-Morosi] Let $(M,\Pi_1,\Pi_2)$ be a  bi-Poisson manifold with $\Pi_1$ symplectic. 
Then, for any polynomial $\rho(z)$, the bivector $\Pi_\rho$ defined by
$$ \Pi_\rho(\alpha,\cdot)=\rho(R)\Pi_1(\alpha,\cdot),\enskip \alpha\in \Omega^1 (M)$$
defines a Poisson structure on $M$, compatible with $\Pi_1$.\hfill $\Box$\label{polPoi}
\end{proposition}


\subsection{Transverse Hilbert schemes on symplectic surfaces\label{symplectic}}

 Beauville's  construction \cite{Beau} of a symplectic form on the Hilbert scheme $S^{[n]}$ of $n$ points in a symplectic surface $S$ has been extended by Bottacin \cite{Bot} to Poisson surfaces: any Poisson structure on a complex surface $S$ induces a Poisson structure on $S^{[n]}$. Therefore, if $(S,\omega)$ is a symplectic surface and $\pi:S\to \cx$ a holomorphic map, we obtain two Poisson bivectors on $S^{[n]}$: $\Pi_1$ induced by $\omega^{-1}$ and $\Pi_2$ induced by $\pi\cdot\omega^{-1}$, where $\omega^{-1}=\#_\omega\omega$ (i.e.\ the bivector dual to $\omega$) and $\pi$ is viewed as a function on $S$. Since the Poisson structures 
on $S$ are compatible, $\Pi_1$ and $\Pi_2$ are compatible (compatibility is trivial on the open dense subset where $D$ consists of distinct points, and hence $[\Pi_1,\Pi_2]$ vanishes everywhere). Observe that the corresponding recursion operator $R$ (cf.\ \S\ref{bi}) is the endomorphism of $TS^{[n]}$ given by the multiplication by $\pi$ on each $T_DS^{[n]}\simeq H^0(D,\sN_{D/S})$ where $\sN_{D/S}$ denotes the normal sheaf of $D$ in $S$. This is the endomorphism considered in \cite{limit,Nicc1, Nicc2}.  The coefficients of its Pfaffian polynomial  define a map $S^{[n]}\to S^n\cx\simeq \cx^n$. Its restriction to the transverse Hilbert scheme $S^{[n]}_\pi$ coincides with the canonical map $\pi^{[n]}$ introduced in \S\ref{trH}. Let us prove the following properties of $S^{[n]}_\pi$ and $\pi^{[n]}$.
\begin{proposition} Let $S$ be a complex symplectic surface with a holomorphic map $\pi:S\to\cx$. Then the transverse Hilbert scheme $S^{[n]}_\pi$ is a nondegenerate bi-Poisson manifold with the following properties:
\begin{itemize}
\item[(i)] the coefficients of the minimal polynomial of the corresponding recursion operator $R$ coincide with the canonical map $\pi^{[n]}:S^{[n]}_\pi\to \oC^{n}$;
\item[(ii)] at any point of its degeneracy locus, the Poisson structure $\Pi_2-\lambda\Pi_1$ has rank $2n-2$ ($\lambda\in \cx$);
\item[(iii)] on the subset of $\pi^{[n]}$-regular points, the Poisson structure $\Pi_2-\lambda\Pi_1$ is log-symplectic for every $\lambda\in \cx$.
\end{itemize}\label{que}
\end{proposition}
\begin{proof} We already know that  $S^{[n]}_\pi$ is a nondegenerate bi-Poisson manifold. Owing to the definition of the recursion operator, we know that the geometric multiplicity of each eigenvalue is even. Now observe that the multiplication by $\pi$ defines also an endomorphism $\bar R$ of $T_{\pi(D)}\cx^{[n]}$. The geometric multiplicity of every eigenvalue of $\bar R$ is equal to $1$ (since $\pi(D)\in S^n\cx$ has length $n$). We also know that the characteristic polynomial  of $\bar R$ is equal to the Pfaffian polynomial $\mu_R(\lambda)$ of $R$, and that the characteristic polynomial of $R$ is $\mu_R(\lambda)^2$. Putting this together, we conclude that the geometric multiplicity of every  eigenvalue of $R$ is equal to $2$ and that the minimal polynomial of $R$ is equal to $\mu_R(\lambda)$. This proves statements (i) and (ii). The third statement follows from Proposition \ref{log}.
\end{proof}
\begin{remark} Statement (i) has been shown in \cite[Remark 2.4]{Nicc1} under the assumption that $\pi$ is a submersion.\end{remark}
\begin{remark} Since, owing to the above mentioned result of Bottacin, any Poisson structure on $S$ induces a Poisson structure on $S^{[n]}$, we can conclude that if $S$ is a Poisson surface with a holomorphic map $\pi:S\to\cx$, then $S^{[n]}$ is a bi-Poisson manifold. The bi-Poisson structure will, however, be degenerate if $S$ is not symplectic.
\end{remark}
\begin{remark} Suppose that $(z,u)$ are Darboux coordinates for the symplectic form $\omega$ on an open subset $U$ of $S$, i.e.\ $\omega=dz\wedge du$ on $U$. Suppose also that $\pi(z,u)=z$ (which implies that $\pi$ is a submersion on $U$). Then the corresponding open subset $U^{[n]}_\pi$ can be described as an open subset of $\{(q(z),p(z))\}$, where $q(z)$ is a monic polynomial of degree $n$ and $p(z)$ is a polynomial of degree at most $n-1$, such that, for every root $z_i$ of $q$, $(z_i,p(z_i))\in U$. On the open dense subset of $U^{[n]}_\pi$, where the roots are distinct, the two Poisson structures are given by:
$$ \Pi_1=\sum_{i=1}^n\frac{\partial}{\partial z_i}\wedge \frac{\partial}{\partial u_i},\quad \Pi_2=\sum_{i=1}^nz_i\frac{\partial}{\partial z_i}\wedge \frac{\partial}{\partial u_i},\enskip \text{where $u_i=p(z_i)$}.$$
\label{local}\end{remark}

We can now characterise transverse Hilbert schemes on symplectic surfaces, in the case when $\pi$ is a submersion (i.e.\ the case originally considered by Atiyah and Hitchin):
\begin{theorem} Let $(M^{2n},\Pi_1,\Pi_2)$ be a holomorphic bi-Poisson manifold with $\Pi_1$ symplectic. Assume that the coefficients of the Pfaffian polynomial of the corresponding recursion operator $R$ define a submersion $p:M\to \oC^{n}$ and that, for each $\lambda\in \cx$, if  the degeneracy locus $D_\lambda$ of $\Pi_2-\lambda\Pi_1$ is nonempty, then its symplectic foliation is simple.
\par
Then there exists a symplectic surface $S$ with a holomorphic submersion $\pi:S\to \cx$ and a local bi-Poisson biholomorphism $\Phi:(M,\Pi_1,\Pi_2)\to S^{[n]}_\pi$.\label{tr-sympl}
\end{theorem}
\begin{proof} Let $\mu_m(\lambda)$ denote the Pfaffian polynomial of $R_{m}$. We consider the following incidence variety (cf.\ \cite[pp.\ 40-43]{AH},\cite{Nicc2}):  $$T=\{(\lambda,m)\in \cx\times M; \, m\in D_\lambda\}=\{(\lambda,m)\in \cx\times M; \, \mu_m(\lambda)=0\}.$$
Due to the assumptions and  to Proposition \ref{log}, $T$ is smooth and the symplectic foliation on each $D_\lambda$ is simple with codimension one leaves. We thus obtain an integrable simple foliation $\sF$ of $T$, the leaf space of which is a $2$-dimensional complex manifold $S$ with a canonical holomorphic submersion $\pi:S\to \cx$. 
\par
On each $D_\lambda$ there is a canonically defined (closed) $1$-form $\alpha_\lambda$ (cf.\ \eqref{normal} and the following lines), the kernel of which corresponds to the symplectic foliation of $D_\lambda$. Thus $\alpha_\lambda\wedge d\lambda$ defines a nondegenerate, hence symplectic,  $2$-form on $S$.

 The scheme-theoretic inverse image of a point $m\in M$ defines a $0$-dimensional subspace $Z_m$ of $T$ with structure sheaf isomorphic to $\cx[\lambda]/(\mu_m(\lambda))$. The projection $T\to \cx$ maps $Z_m$ isomorphically onto a $0$-dimensional subspace of $\cx$ of length $n$. Thus $Z_m$ descends to an  element of $S^{[n]}_\pi$, and we obtain a holomorphic map $\Phi:M\to S^{[n]}_\pi$. It remains to show that $\Phi$ is a local diffeomorphism. Since the coefficients of $\mu_m$ define a submersion, the corresponding Hamiltonian vector fields do not vanish anywhere. Near any point $p\in M$ we have therefore the ``action-angle" coordinates on a neighbourhood $U$ (given by coefficients of $\mu$ and the local free action of $\cx^n$). Let $S_p$ be the symplectic surface obtained from $U$ by the above procedure. On $(S_p)^{[n]}_\pi$ there are analogous ``action-angle" coordinate and therefore we obtain a holomorphic map
 $\psi:(S_p)^{[n]}_\pi\to U$. Fernandes \cite{Rui} shows that, on the subset where the eigenvalues are distinct, there exist local coordinates $z_i,u_i$ such that
$$ \Pi_1=\sum_{i=1}^n \frac{\partial}{dz_i}\wedge \frac{\partial}{du_i},\quad \Pi_2=\sum_{i=1}^n z_i\frac{\partial}{dz_i}\wedge \frac{\partial}{du_i}.$$ Thus, owing to Remark \ref{local}, $\psi$ is the inverse of $\Phi$ on the open dense subset where the roots of $\mu_m$ are distinct (and a bi-Poisson isomorphism) and,  hence, $\psi$ is the inverse of $\Phi_{|U}$.                                                                                                                                                                                            
\end{proof}
\begin{remark} Presumably the result remains true without the assumption that the symplectic foliations of $D_\lambda$ are simple, provided  we replace ``symplectic surface" with ``$2$-dimensional symplectic stack".\end{remark}

\section{Hyperk\"ahler geometry of transverse Hilbert schemes}
\subsection{Ward transform}
Let us briefly recall the essential features of the Ward transform \cite{Ward,Sal} in the case of hypercomplex manifolds. Let $Z$ be complex manifold with a surjective holomorphic submersion $\pi:Z\to \oP^1$, and let $M^\cx$ be the Kodaira moduli space of sections with normal bundle isomorphic to $\sO(1)^{\oplus n}$. The twistor double fibration  in this case is simply
\begin{equation} M^\cx\stackrel{\tau}{\longleftarrow} M^\cx\times \oP^1\stackrel{\nu}{\longrightarrow}Z.\label{double}
\end{equation}
If $F$ is an $M^\cx$-uniform (i.e.\ $h^0(\nu(\tau^{-1}(m)))$ is constant on $M^\cx$) holomorphic vector bundle on $Z$, then we obtain an induced holomorphic vector bundle $\hat F=\tau_\ast\nu^\ast F$ on $M^\cx$. In particular, if we denote by $E$ the vector bundle induced from $T_\pi Z\otimes \pi^\ast\sO(-1)$ (where  $T_\pi Z=\ker d\pi$) and by $H$ the trivial vector bundle with fibre $\cx^2$, we have $TM\simeq E\otimes H$. Furthermore, the vector bundle induced from $\pi^\ast \sO(k)$, $k\geq 0$, is simply $S^k H$, and if $F$ is $M^\cx$-trivial (i.e.\ trivial on each line $ \nu(\tau^{-1}(m))$), then the bundle induced from $F\otimes \pi^\ast \sO(k)$ is $\hat F\otimes S^k H$, for any $k\geq 0$.
\par
Recall also that an induced vector bundle comes equipped with a first-order differential operator, which arises as the pushforward of a partial connection on $\nu^\ast F$, which is basically the exterior derivative in the fibre directions. If $F$ is $M^\cx$-trivial, then this operator is a linear connection on $\hat F$. We can identify this operator for bundles of the form $\hat F\otimes S^k H$, which are induced from $F\otimes \pi^\ast \sO(k)$, where $F$ is $M^\cx$-trivial. We fix an isomorphism $H\simeq H^\ast$ (which corresponds to a choice of isomorphism $H^1(\oP^1,\sO(-2))\simeq \cx$). We denote by $\alpha$ the natural projection $S^k H\otimes H\to S^{k+1} H$ (which corresponds to multiplication of sections of $\sO(k)$ and of $\sO(1))$, and use the same letter for the corresponding map on $W\otimes S^k H\otimes H\to W\otimes S^{k+1} H$ for any holomorphic vector bundle $W$. The induced differential operators ${\rm D}$ are then:
\begin{itemize}
\item[1.] on $S^k H$, ${\rm D}=\alpha\circ d$;
\item[2.] on $\hat F\otimes S^k H$,
    $ {\rm D}=\alpha\circ (\nabla\otimes {\rm Id}) + {\rm Id}\otimes (\alpha\circ d)$, where $F$ is $M^\cx$-trivial and $\nabla$ denotes the induced connection on $\hat F$. The principal symbol of ${\rm D}$ is $\alpha$.
    \end{itemize}
The construction of $D$ as the push-forward of a partial connection shows, in particular, that ${\rm D}s=0$ if and only if $s=\tau_\ast\eta^\ast \tilde s$ for a holomorphic section $\tilde s$ of $F(k)$ on $Z$.

\subsection{Hyper-Poisson bivectors\label{HP}}
Let $M$ be a hypercomplex manifold with twistor space $Z$. As discussed above, the vector bundle $T_\pi Z\otimes \pi^\ast\sO(-1)$ is $M$-trivial, and hence the operator ${\rm D}$ on the induced vector bundle $E$ is a linear connection $\nabla$. Recall that the tensor product of $\nabla$ and the flat connection on $H$ is a torsion-free linear connection on $TM$ known as the {\em Obata connection}. The induced operator ${\rm D}$ on $TM$ is therefore the composition of the Obata connection and the projection $H\otimes H\to S^2H$. Similarly, the vector bundle $\Lambda^r (T_\pi Z\otimes \pi^\ast\sO(-1))$ is $M$-trivial, and hence the vector bundle on $M$ induced from $\Lambda^r T_\pi Z$ is $\Lambda^r E\otimes S^r H$. This is a direct summand of $\Lambda^r T^\cx M$, which Salamon \cite[Prop.\ 4.2]{Sal} identifies with the subspace of finite linear combinations of multivectors of type $(r,0)$ for different complex structures.
Salamon also shows that its $Sp(1)$-invariant complement is
\begin{equation} B^r=\bigcap_{\zeta}\bigl( T^{1,r-1}_\zeta M\oplus T^{2,r-2}_\zeta M\oplus \dots\oplus T^{r-1,1}_\zeta M\bigr),\label{Br}\end{equation}
where $\zeta\in \oP^1$ labels different complex structures of the hypercomplex structure.
As explained in the previous subsection,  a multivector field $\Pi\in \Gamma(\Lambda^r E\otimes S^r H)$ arises from a holomorphic section of $\Lambda^r T_\pi Z$
precisely if it satisfies the equation ${\rm D}\Pi=0$. In that case, for any complex structure $I_\zeta$, the $(r,0)$-component of $\Pi$ is the corresponding holomorphic multivector field on the fibre $\pi^{-1}(\zeta)$ of $Z$.
\begin{definition} Let $M$ be a hypercomplex manifold. A bivector  $\Pi\in  \Gamma(\Lambda^2 T M)$ is called a {\em hyper-Poisson bivector} if
\begin{itemize}
\item[(i)] $\Pi\in \Gamma(\Lambda^2 E\otimes S^2 H)$;
\item[(ii)] for each complex structure $I_\zeta$, the corresponding $(2,0)$-component $\Pi_\zeta^{2,0}$ of $\Pi$ is a holomorphic Poisson bivector on $(M,I_\zeta)$.
\end{itemize}\label{hyperP}
\end{definition}
\begin{remark} As explained above, the condition that $\Pi_\zeta^{2,0}$ is holomorphic for each $\zeta$ is equivalent to ${\rm D}\Pi=0$. On the other hand, condition (ii) implies that $[\Pi,\Pi]\in \Gamma(B^3)$, where $[\;,\;]$ denotes the Schouten-Nijenhuis bracket and  $B^3$ is defined in \eqref{Br}.
\end{remark}
As usual, given a bivector field on $M$, we can define define a bracket of (real- or complex-valued) functions on $M$ by
\begin{equation}\{f,g\}=\Pi(df,dg).\label{bracket}\end{equation}
The name ``hyper-Poisson" is justified by the following observation, which follows directly from the definition.
\begin{proposition} A bivector $\Pi\in \Gamma(\Lambda^2 E\otimes S^2 H)$ is hyper-Poisson if and only if,  for each complex structure $I_\zeta$, the bracket \eqref{bracket} is a Poisson bracket on the sheaf $\sO(M,I_\zeta)$ of $I_\zeta$-holomorphic functions.\hfill $\Box$
\end{proposition}
\begin{remark} Our notion of ``hyper-Poisson" is different from \cite{HS}. There, it means a triple $(\pi_1,\pi_2,\pi_3)$ of bivectors, such that $\pi_2-i\pi_3$ is an $I_1$-holomorphic Poisson bivector etc. We do not think there is a danger of confusion, since we talk about {\em hyper-Poisson bivectors}, while \cite{HS} deals with {\em hyper-Poisson triples}.
\end{remark}
\begin{definition} Let $M$ be a hyperk\"ahler manifold. A hyper-Poisson bivector  $\Pi$ on $M$ is said to be compatible with the hyperk\"ahler structure if, for every complex structure, the holomorphic Poisson bivector $\Pi_\zeta^{2,0}$ is compatible with $\Omega_\zeta^{-1}$, where $\Omega_\zeta$ is the corresponding parallel holomorphic symplectic form.
\end{definition}
\begin{example} Recall that the twistor space $Z$ of a hyperk\"ahler manifold is equipped with a fibrewise $\sO(2)$-valued complex symplectic form $\omega$, i.e. a section of $\Lambda^2T^\ast_\pi Z\otimes \sO(2)$. It can be viewed as a (holomorphic) section of $\Lambda^2 (T_\pi Z(-1))^\ast$, i.e.\ it induces a symplectic form on the bundle $E$. The fibrewise bivector $\omega^{-1}$ is a section of $\Lambda^2 (T_\pi Z(-1))$ and so multiplying it by a real section of $\pi^\ast\sO(2)$ yields a holomorphic section of $\Lambda^2 T_\pi Z$ compatible with the real structure, i.e.\ a hyper-Poisson bivector. This bivector is simply a constant multiple of $\omega_\alpha^{-1}$, where $\omega_\alpha$ is one of  the  K\"ahler forms of the hyperk\"ahler metric ($\alpha$ is determined by the chosen section of $\sO(2)$). In other words, for  any K\"ahler form $\omega_\alpha$ of the hyperk\"ahler metric, the dual bivector $\omega_\alpha^{-1}$ is a hyper-Poisson bivector compatible with the hyperk\"ahler structure. The corresponding Poisson bracket on $\sO(M,I_\zeta)$ is identically $0$ when $I_\zeta=\pm I_\alpha$.
\end{example}

We can classify hyper-Poisson bivectors on $4$-dimensional hyperk\"ahler manifolds.
\begin{theorem} Let $(M,g,I_1,I_2,I_3)$ be a $4$-dimensional hyperk\"ahler manifold with corresponding K\"ahler forms $\omega_1,\omega_2,\omega_3$. A bivector $\Pi$ on $M$ is hyper-Poisson if and only if 
\begin{equation}\Pi=f_1\omega_1^{-1}+f_2\omega_2^{-1}+f_3\omega_3^{-1},\label{Pi}\end{equation}
for smooth functions $f_1,f_2,f_3:M\to \oR$ satisfying $I_1df_1=I_2df_2=I_3df_3$. Such a bivector is compatible with the hyperk\"ahler structure.
\label{dim=4}\end{theorem}
\begin{proof} If dim $M=4$, then $\Lambda^2 E \otimes S^2 H$ is spanned at each point by $\omega_1^{-1},\omega_2^{-1},\omega_3^{-1}$. Therefore $\Pi$ must be of the form \eqref{Pi}.
Its $(2,0)$-component with respect to  $I_1$ is equal to 
$$\frac{1}{2}(f_2+if_3)(\omega^{-1}_2-i\omega_3^{-1})=2(f_2+if_3)(\omega_2+i\omega_3)^{-1},$$
Hence, if $\Pi$ is hyper-Poisson, then $f_2+if_3$ is $I_1$-holomorphic (since $ \omega_2+i\omega_3$ is $I_1$-holomorphic). Similarly, $f_3+if_1$ must be $I_2$-holomorphic, and $f_1+if_2$ must be $I_3$-holomorphic. This triple of conditions is equivalent to $I_1df_1=I_2df_2=I_3df_3$. Conversely, if the latter condition holds, then, for any complex structure, the $(2,0)$-part of $\Pi$ is holomorphic. The $(2,0)$-part is also Poisson, since $\dim_\cx M=2$. For the same reason $\Pi$ is compatible with the hyperk\"ahler structure. 
\end{proof}
\begin{remark} This result can be, of course, also proved by describing holomorphic sections of the line bundle $\Lambda^2T_\pi Z$.\end{remark}
\begin{corollary} A $4$-dimensional hyperk\"ahler manifold $M$ admits a hyper-Poisson bivector, other than a constant linear combination of $\omega_1^{-1},\omega_2^{-1},\omega_3^{-1}$, if and only if $M$ admits a non-zero tri-Hamiltonian vector field. In this case the functions $f_1,f_2,f_3$ are the three moment maps for this vector field.\hfill$\Box$\end{corollary}

Let $M^{4d}$ be a hyperk\"ahler manifold equipped with a compatible hyper-Poisson bivector $\Pi$. For each complex structure $I_\zeta$, $\zeta\in \oP^1$, we have the compatible holomorphic Poisson bivectors $\Pi_1=\Omega_\zeta^{-1}$ and $\Pi_2=\Pi_\zeta^{2,0}$. The Pfaffian polynomials of the corresponding recursion operators combine to define a polynomial $p(\zeta,\lambda)$ of the form
\begin{equation} p(\zeta,\eta)=\lambda^d+p_1(\zeta)\lambda^{d-1}+\dots+p_d(\zeta),\label{pol}\end{equation}
where the degree of $p_i$ is $2i$. Each $p_i$ defines a section of $\sO(2i)$ on $\oP^1$, and \eqref{pol} can be viewed as a holomorphic map 
$$Z\longrightarrow \bigoplus_{i=1}^{d}\bigl|\sO(2i)\bigr|,$$
where $\bigl|\sO(2i)\bigr|$ denotes the total space of $\sO(2i)$. This map is compatible with the real structures. 
In particular $p_1(\zeta)$ is a quadratic polynomial, compatible with the real structure of $|\sO(2)|$, and hence $p_1$ is the hyperk\"ahler moment map for a tri-Hamiltonian vector field $X_\Pi$. We shall call $X_\Pi$ the {\em canonical Killing vector field}.

\subsection{Hyperk\"ahler transverse Hilbert schemes\label{Hts}}
Let $M$ be a $4$-dimensional hyperk\"ahler manifold with a non-trivial tri-Hamiltonian Killing vector field. The moment map induces a holomorphic map $\mu$ from the twistor space $Z$ of $M$ to $|\sO(2)|$. Following Atiyah and Hitchin \cite{AH} we can perform the transverse Hilbert scheme construction on fibres of $Z\to \oP^1$ and obtain a new twistor space $Z_\mu^{[d]}$. Sections of   $Z_\mu^{[d]}\to \oP^1$ are in $1-1$-correspondence with  $1$-dimensional compact complex subspaces $C$ of $Z$ such that:
\begin{itemize}
\item[(i)] the projection $\pi:C\to \oP^1$ is flat with fibres of length $d$;
\item[(ii)] the projection $\mu$ induces a scheme-theoretic isomorphism between $C$ and and its image in $|\sO(2)|$.
\end{itemize}
Observe, that given (ii), (i) simply means that $\mu(C)$ is defined by $p(\zeta,\lambda)=0$, where $p$ is as in \eqref{pol}.
\par
Furthermore, as explained in \cite{Sigma}, the normal bundle of the section of $Z_\mu^{[d]}$ corresponding to $C$ splits as $\sO(1)^{\oplus 2d}$ if and only if the normal sheaf $\sN_{C/Z}$ of the curve $C$ in $Z$ satisfies $H^\ast(C,\sN_{C/Z}(-2))=0$. On the Kodaira moduli space of such sections, satisfying in addition a reality condition, we obtain again a (pseudo)-hyperk\"ahler metric. We shall denote this hyperk\"ahler manifold by $M_\mu^{[d]}$ and refer to it as a {\em hyperk\"ahler transverse Hilbert scheme}.

Consider now the $\sO(2)$-valued complex symplectic form $\omega$ on the fibres of $Z$, which can be viewed  as a (holomorphic) section of $\Lambda^2 (T_\pi Z(-1))^\ast$. Performing the construction of section \ref{symplectic} fibrewise on $Z$ yields a fibrewise Poisson structure $\Pi_2$ on $Z_\mu^{[d]}$, which is a section of $\Lambda^2 (T_\pi Z_\mu^{[d]}(-1))\otimes \sO(2)\simeq \Lambda^2 T_\pi Z_\mu^{[d]}$, i.e.\ it induces a bivector $\hat\Pi_2\in \Lambda^2 E\otimes S^2 H$ on $M^{[d]}_\mu$. We shall see shortly that it is a purely imaginary bivector and, hence, $\Pi=-i\hat \Pi_2$ is real and, as discussed above, a hyper-Poisson bivector. Moreover, for each $\zeta\in \oP^1$, its $(2,0)$-component $\Pi_\zeta^{2,0}$ is compatible with $\Omega_\zeta^{-1}$, since $\Pi_2$ restricted to the fibre $Z_\zeta$ is compatible with $\omega_\zeta^{-1}$. Thus $\Pi$ is a hyper-Poisson bivector compatible 
with the hyperk\"ahler structure on $M_\mu^{[d]}$.
\begin{example} Consider the case $d=1$. The $(2,0)$-part of $\hat\Pi_2$ is equal to $\mu\Omega^{-1}$, where $\Omega$ is the corresponding parallel holomorphic $2$-form and $\mu$ is the corresponding holomorphic moment map. Comparing with Theorem \ref{dim=4}, we conclude that
\begin{equation*}2\hat\Pi_2=\mu_1\omega_1^{-1}+\mu_2\omega_2^{-1}+\mu_3\omega_3^{-1},\label{Pi2}\end{equation*}
where $\omega_1,\omega_2,\omega_3$ denote the K\"ahler forms for complex structures $I_1,I_2,I_3$ and $\mu_1,\mu_2,\mu_3$ are the corresponding  moment maps.
In particular, $\hat\Pi_2$ is purely imaginary. Since the real structure on $Z_\mu^{[d]}$ is induced from the one on $Z$, it follows that  $\hat\Pi_2$  on $M_\mu^{[d]}$ is also purely imaginary.                                                                    
\label{d=1}\end{example}
\begin{remark} Let $Z$ be a complex $3$-fold with a holomorphic map $\mu:Z\to |\sO(2)|$ such that the composite map to $\oP^1$ is surjective. Suppose further that $Z$ has a fibrewise $\sO(2)$-valued complex symplectic form $\omega$ and a real structure covering the natural real structure on $|\sO(2)|$. In other words, $Z$ fulfills all conditions of the twistor space of a hyperk\"ahler $4$-manifold with a tri-Hamiltonian Killing vector field, except the existence of sections with normal bundle $\sO(1)\oplus\sO(1)$. In principle, it could happen that $Z$ contains curves of degree $d>1$, but not of degree $1$ (although we do not know such an example). In this case $M_\mu^{[d]}$ is still well defined and a (pseudo)-hyperk\"ahler manifold, although $M$ does not exist. We shall still call
$M_\mu^{[d]}$ a  hyperk\"ahler transverse Hilbert scheme, since the construction requires only the existence of $Z$, not necessarily of $M$.\label{virtual} \end{remark}

We can characterise hyperk\"ahler transverse Hilbert schemes arising from $4$-manifolds with a locally free tri-Hamiltonian $\oR$-action as follows:
\begin{theorem} Let $M^{4d}$ be a hyperk\"ahler manifold, equipped with a compatible hyper-Poisson bivector $\Pi$, such that for every complex structure $I_\zeta$, the holomorphic Poisson bivectors $\Pi_1=\Omega_\zeta^{-1}$ and $\Pi_2=\Pi_\zeta^{2,0}$ satisfy the assumptions of Theorem \ref{tr-sympl}. Then there exists a complex $3$-fold $Z$ with properties listed in Remark \ref{virtual} such that $M^{4d}$ is locally isomorphic to $M_\mu^{[d]}$ as a hyper-Poisson hyperk\"ahler manifold. In addition, the holomorphic map $\mu:Z\to |\sO(2)|$ is a submersion.
\label{tr-hk}\end{theorem}
\begin{proof} We can perform the construction in the proof of Theorem \ref{tr-sympl} fibrewise on the twistor space of $M^{4d}$ and obtain $Z$. Its properties follow easily.
\end{proof}
\begin{remark} This theorem remains true for pseudo-hyperk\"ahler $M^{4d}$. We do not know whether the induced metric on a hyperk\"ahler transverse Hilbert scheme is always positive definite.
\end{remark}
\begin{remark} The canonical Killing field $X_\Pi$ is transverse to the foliation defined in the proof of Theorem \ref{tr-sympl} (on each fibre of the twistor space). Therefore the vertical vector field on $Z$, which gives the projection to $|\sO(2)|$ is induced by $X_\Pi$. In particular, if $X_\Pi$ integrates to an action of $\oR$ or $S^1$ on $M^{4d}$ and the resulting $3$-fold $Z$ admits sections, then $Z$  is the twistor of a hyperk\"ahler $4$-manifold with a tri-Hamiltonian action of $\oR$ or $S^1$.
\end{remark}

\begin{remark} One can can consider, more generally, hyperk\"ahler transverse Hilbert schemes on $4$-manifolds $M$, the twistor space of which maps to $|\sO(2r)|$, $r>1$, rather than to $|\sO(2)|$. The corresponding object arising on $M_\mu^{[d]}$ is then a section  $\Pi$ of $\Lambda^2 E\otimes S^{2r} H$, satisfying ${\rm D}\Pi=0$. If we view $\Pi$ as a $\Lambda^2 E$-valued polynomial $\Pi(\zeta)$ of degree $2r$, then, for each $\zeta\in \oP^1$, $\Pi(\zeta)$ defines an $I_\zeta$-holomorphic Poisson bivector on $M_\mu^{[d]}$, compatible with $\Omega_\zeta^{-1}$.  Theorem \ref{tr-hk} remains true, as do the results of the next subsection (with obvious modifications). \label{tensor} 
\end{remark} 

\subsection{Linear geometry of quaternionic bivectors\label{linear}} 

Let $V$ be a real vector space of dimension $4n$, equipped with the standard flat quaternionic structure $(g,I_1,I_2,I_3)$. We denote the corresponding (linear) symplectic forms by $\omega_1,\omega_2,\omega_3$. The complexification $V^\cx$ decomposes as $E\otimes H$, where $E$ and $H$ have complex dimensions $2n$ and $n$, respectively, and are equipped with the standard quaternionic-Hermitian structure, i.e.\ complex symplectic forms $\omega_E$, $\omega_H$ and quaternionic structures $\sigma_E,\sigma_H$, so that $\omega_E(x,\sigma_E(x))>0$ and similarly for $H$.
\par
Let $\Pi$ be a bivector in $\Lambda^2 V$ belonging to $\Lambda^2 E\otimes S^2 H$. We define an endomorphism $A:V\to V$ by $ \#_\Pi\circ \#_g^{-1}$. Since $\Pi\in \Lambda^2 E\otimes S^2 H$, it is an eigenvector of the operator $\sum_{i=1}^3 I_i\otimes I_i$ with eigenvalue $-1$. Thus $A$ is an eigenvector of $\sum_{i=1}^3 I_i\otimes I_i$  on $V^\ast \otimes V$ with eigenvalue $1$.  Haydys \cite{Hay} calls such endomorphisms {\em aquaternionic} and shows that they are of the form $A=I_1A_1+I_2A_2+I_3A_3$, where each $A_i$ is quaternionic, i.e.\ it commutes with $I_1,I_2,I_3$. Since $\Pi$ is antisymmetric and real, the $A_i$ are quaternionic matrices which are quaternion-Hermitian, i.e.\ $A_i^\dagger=A_i$, where $\dagger$ denotes the quaternionic adjoint. In terms of the symplectic structures (cf.\ Example \ref{d=1}).
\begin{equation*}\Pi(\alpha,\cdot)=A_1\omega_1^{-1}(\alpha,\cdot)+A_2\omega_2^{-1}(\alpha,\cdot)+A_3\omega_3^{-1}(\alpha,\cdot).
\label{PiA}\end{equation*}
It follows that the $(2,0)$ component of $\Pi$ for the complex structure $I_1$ is 
$$ \Pi^{2,0}(\alpha,\cdot)=(A_2+iA_3)(\omega_2+i\omega_3)^{-1}(\alpha,\cdot),$$
and similarly for other complex structures. Viewing the $A_i$ as endomorphisms of $E$ (which corresponds to the canonical homomorphism  $\gl(n,\oH)\hookrightarrow \gl(2n,\cx)$) we obtain a quadratic endomorphism
$$A(\zeta)=(A_2+iA_3)+2iA_1\zeta-(A_2-iA_3)\zeta^2$$
of $E$, where each $A_i\in \gl(2n,\cx)$ is symmetric with respect to the symplectic form $\omega_{E}$.
We can consider the sheaf morphism $\eta-A(\zeta):E\otimes \sO(-2)\to E$ on $T\oP^1$. Its cokernel, which we denote by $\sF$, is a $1$-dimensional sheaf. Since each $A(\zeta)$ is symmetric with respect to the standard symplectic form on $\cx^{2n}$, the characteristic polynomial of $A(\zeta)$ is of the form $p(\zeta,\eta)^2$, where $p(\zeta,\eta)$ is a polynomial of degree $n$ in $\eta$. We call the scheme $C=\{(\zeta,\eta);p(\zeta,\eta)=0\}$ as the {\em spectral curve of the bivector $\Pi$} and view the sheaf $\sF$ as being supported on $C$.  If $C$ is smooth, then $\sF$ is a rank $2$ vector bundle with $\det \sF\simeq K_C(2)$ (cf.\ \cite{Beau2}, which contains more results on vector bundles arising this way). If the bivector $\Pi$ arises via the hyperk\"ahler transverse Hilbert scheme construction from a $3$-dimensional twistor space $Z$, then the spectral curve $C$ is precisely the curve in $Z$ corresponding to a point in $M_\mu^{[n]}$ and the sheaf $\sF$ is isomorphic to $\sN_{C/Z}(-1)$.
\begin{remark} One can associate to $\Pi$ another spectral object. Since the $A_i$ are quaternion-Hermitian matrices, they are diagonalisable over $\oH$ with real eigenvalues. Let us denote the product of eigenvalues of a quaternion-Hermitian matrix $X$ by $\det_{\oH} X$ (the so-called {\em Moore determinant}). We can define a surface in $\oR P^3$ as
$$S_{\oR}=\left\{[x_0,x_1,x_2,x_3]\in \oR P^3; \; \det{}_{\oH}\bigl(x_0-x_1A_1-x_2A_2-x_3A_3\bigr)=0\right\}.$$
It is a (ramified) $n$-fold of $\oR P^2$ (via $[x_0,x_1,x_2,x_3]\to [x_1,x_2,x_3]$), and its  complexification is a surface $S$ in $\cx P^3$ which can be defined as 
$$ S=\left\{[z_0,z_1,z_2,z_3]\in \oC P^3; \; \det\bigl(z_0-z_1A_1-z_2A_2-z_3A_3\bigr)=0\right\},$$
where the $A_i$ are now complex $2n\times 2n$ matrices via the  homomorphism  $\gl(n,\oH)\hookrightarrow \gl(2n,\cx)$.  The   intersection of $S$ with the quadratic cone 
$$ x_1=2i\zeta, \enskip x_2=1-\zeta^2,\enskip x_3=i(1+\zeta^2)$$
is the doubled spectral curve $C$.
The sheaf $\sF$ extends to $S$ and is defined as the cokernel of 
$$ z_0-z_1A_1-z_2A_2-z_3A_3: E\otimes \sO(-2)\to E$$ on $\oP^3$. At present we do not understand the significance of $S_\oR$ (as opposed to $C$) for the geometry of hyper-Poisson manifolds.\end{remark}

\section{The hyper-Poisson bivector of the monopole moduli space}

We consider the moduli space $\sM_k$ of $SU(2)$-monopoles of charge $k$, described as the moduli space of $\fU(k)$-valued solutions of Nahm's equations on $(0,2)$, with simple poles at $t=0,2$ and residues defining the standard irreducible representation of $\su(2)$. The Nahm equations are $\dot{T}_1=[T_1,T_0]+[T_2,T_3]$ and two further equations, obtained by cyclic permutations of indices $1,2,3$. The tangent space at $[T_0,T_1,T_2,T_3]$ is given by quadruples of smooth maps $(t_0,t_1,t_2,t_3)$ from $[0,2]$ to $\fU(k)$ satisfying the equations
\begin{eqnarray*} \dot{t}_0 & = & [t_0,T_0]+ [t_1,T_1]+ [t_2,T_2]+ [t_3,T_3]\\
\dot{t_1} & = & [T_1,t_0]+[t_1,T_0]+[T_2,t_3]+[t_2,T_3]\\\
\dot{t_2} & = & [T_2,t_0]+[t_2,T_0]+[T_3,t_1]+[t_3,T_1]\\
\dot{t_3} & = & [T_3,t_0]+[t_3,T_0]+[T_1,t_2]+[t_1,T_2].
\end{eqnarray*}
The hypercomplex structure is given by the right multiplication by quaternions on $t_0+t_1i+t_2j+t_3k$ and the Riemannian metric $g$ is 
$$ \|(t_0,t_1,t_2,t_3)\|^2=-\int_0^2\tr(t_0^2+t_1^2+t_2^2+t_3^2).
$$
As explained by Atiyah and Hitchin in \cite[Ch.6]{AH}, the hyperk\"ahler manifold $\sM_k$ is the hyperk\"ahler transverse Hilbert scheme associated to $S^1\times \oR^3$.
Thus, according to \S\ref{Hts}, it posseses a natural hyper-Poisson bivector compatible with the hyperk\"ahler structure. We can identify this bivector as follows:
\begin{theorem} The natural hyper-Poisson bivector $\Pi$ on the moduli space $\sM_k$ is given by
\begin{equation*} \#_g^{-1}\Pi = \frac{-i}{4}\int_0^2\tr\Bigl(\sum_{i=1}^3T_i\bigl(dT_i\wedge dT_0-dT_0\wedge dT_i\bigr)+
\sum_{i,j,k=1}^3\epsilon_{ijk}T_idT_j\wedge dT_k\Bigr),\label{formula}\end{equation*}
where $(\phi\wedge \psi)(a,b)=\phi(a)\psi(b)-\phi(b)\psi(a)$.
\label{lotsofT}
\end{theorem}
\begin{remark} It will follow from the proof that the integral is finite.
\end{remark}
\begin{proof} We first compute the K\"ahler forms $\omega_2$ and $\omega_3$ corresponding to the complex structures $J$ and $K$:
$$\omega_2=-\int_0^2\tr\Bigl(dT_0\wedge dT_2+dT_1\wedge dT_3),\quad \omega_3=-\int_0^2\tr\Bigl(dT_0\wedge dT_3+dT_2\wedge dT_1).$$
The $I$-holomorphic $2$-form $\omega_2+i\omega_3$ is therefore given by  
$$\omega_2+i\omega_3=-\int_0^2\tr d(T_0-iT_1)\wedge d(T_2+iT_3).$$
We can now rewrite the integrand in  the formula for $\#_g^{-1}\Pi$ as follows:
\begin{multline*} 
\frac{1}{2}\tr\bigl((T_2+iT_3)(\Phi_2-i\Phi_3)+\\+2T_1(dT_1\wedge dT_0-dT_0\wedge dT_1+dT_2\wedge dT_3-dT_3\wedge dT_2)+\\+(T_2-iT_3)(\Phi_2+i\Phi_3)\bigr),
\end{multline*}
where $$\Phi_2=dT_2\wedge dT_0-dT_0\wedge dT_2+dT_3\wedge dT_1-dT_1\wedge dT_3,$$ $$\Phi_3=dT_3\wedge dT_0-dT_0\wedge dT_3+dT_1\wedge dT_2-dT_2\wedge dT_1.$$
Observe that the first summand is of type $(0,2)$ for the complex structure $I$, the second one of type $(1,1)$, and the third one of type $(2,0)$. Since $\#_g$ exchanges $(2,0)$
and $(0,2)$, we conclude that 
\begin{multline*}\#_g^{-1}\Pi^{2,0}=\frac{-i}{8}\int_0^2\tr(T_2+iT_3)(\Phi_2-i\Phi_3)=\\ =\frac{-i}{8}\int_0^2\tr(T_2+iT_3) \bigl(d(T_2-iT_3)\wedge d(T_0+iT_1)-d(T_0+iT_1)\wedge d(T_2-iT_3)\bigr).\end{multline*}
We now observe that for any bivector $\pi$ and $\Omega=\omega_2+i\omega_3$:
\begin{equation} \Bigl(\#_\Omega^{-1}\pi\Bigr)(u,v)=\Bigl(\#_g^{-1}\pi\Bigr)(Ju+iKu,Jv+iKv).\label{key}\end{equation}
Computing this for $\pi=\Pi^{2,0}$ we obtain:
\begin{equation*}\#_\Omega^{-1}\Pi^{2,0}=\frac{-i}{2}\int_0^2\tr(T_2+iT_3)\bigl(d(T_2+iT_3)\wedge d(T_0-iT_1)-d(T_0-iT_1)\wedge d(T_2+iT_3)\bigr).\label{almost}
\end{equation*}
Thus, if we set $\beta=T_2+iT_3$ and $\alpha=T_0-iT_1$, we obtain
\begin{equation}i\#_\Omega^{-1}\Pi^{2,0}= \frac{1}{2}\int_0^2\tr d\bigl(\beta^2\bigr)\wedge d\alpha.\label{closed}\end{equation}
The complex Nahm equation is the Lax equation $\dot\beta=[\beta,\alpha]$. It follows that $(\beta^2,\alpha)$ also satisfies the Lax equation. Acting by a singular complex gauge transformation which makes $\alpha$ equal to zero (and, consequently, $\beta$ constant) implies that
$$i\#_\Omega^{-1}\Pi^{2,0}=\frac{1}{2}\sum d(\beta_j)^2\wedge \frac{dp_j}{p_j}=\sum\beta_jd\beta_j\wedge \frac{dp_j}{p_j},
$$
where  $\beta_j$ are the poles and $p_j$ the values of the numerator of the rational map corresponding to the given monopole (and the complex structure $I$). This means that $\Pi^{2,0}$ is precisely the holomorphic bivector obtained from the transverse Hilbert scheme construction applied to $\cx\times \cx^\ast$ with the symplectic form $d\beta \wedge \frac{dp}{p}$.
\par
Observe now that, since $\#_g^{-1}\Pi^{2,0}=\bigl(\#_g^{-1}\Pi\bigr)^{0,2}$ and $Ju+iKu=(Ju)+iI(Ju)$, the formula \eqref{key} implies that 
\begin{equation*}\bigl(\#_g^{-1}\Pi\bigr)^{0,2}(u,v)=\#_\Omega^{-1}\Pi^{2,0}(-Ju,-Jv).\label{yes}\end{equation*}
Therefore the integral  defining $\bigl(\#_g^{-1}\Pi\bigr)^{0,2}$ is finite. Similarly, the integral  defining $\bigl(\#_g^{-1}\Pi\bigr)^{2,0}$ is finite. Since $\bigl(\#_g^{-1}\Pi\bigr)^{0,2}+\bigl(\#_g^{-1}\Pi\bigr)^{2,0}$  is the sum of all terms of the form $T_2\cdot \phi$ and $T_3\cdot \phi$ in the formula in the statement, repeating this decomposition for  the complex structure $J$ or $K$ shows that the whole integral in the statement is finite. This also shows that
$$\Pi=\frac{1}{2}\Bigl(\Pi^{2,0}_{I}+\Pi^{2,0}_{-I}+\Pi^{2,0}_{J}+\Pi^{2,0}_{-J}+\Pi^{2,0}_{K}+\Pi^{2,0}_{-K}\Bigr),$$
and so, owing to Lemma 4.2 in \cite{Sal}, $\Pi\in \Gamma(\Lambda^2E\otimes S^2H)$. 
Finally, observe that $\Pi$ is real.
\end{proof}
\begin{remark} The canonical Killing vector field $X_\Pi$ (see \S\ref{HP}) on the moduli space $\sM_k$ is $(t_0,t_1,t_2,t_3)=(i,0,0,0)$. We obtain
$$ i(X_\Pi)\#_g^{-1}\Pi=-\frac{1}{2}\int_0^2\tr\sum_{i=1}^3T_idT_i=-\frac{1}{4}dF,$$
where
$$F=\int_0^2\Bigl(\tr\sum_{i=1}^3 T_i^2+\frac{k(k^2-1)}{4}\bigl(s^{-2}+(s-2)^{-2}\bigr)\Bigr).$$
The function $F$ has been shown by Hitchin \cite{Hit} to essentially give K\"ahler potentials of the monopole metric: for every complex structure, the sum of $F$ and some linear combination of the coefficients of the spectral curve is a K\"ahler potential for the corresponding K\"ahler form. Is this true on a general hyperk\"ahler transverse Hilbert scheme $M_\mu^{[d]}$, i.e.\ is the $1$-form $ i(X_\Pi)\#_g^{-1}\Pi$ similarly related to K\"ahler potentials on  $M_\mu^{[d]}$?
\end{remark}


\begin{thebibliography}{99}

\bibitem{AH}
{ M.F. Atiyah  \and N.J. Hitchin}, {\em The geometry and dynamics of magnetic monopoles}, Princeton University Press,  Princeton (1988).

\bibitem{Beau}
{A. Beauville}, `Vari\'et\'es K\"ahleriennes dont la premi\`ere classe de Chern est nulle', {\it J. Differential Geom.} 18 (1983), no. 4, 755--782.

\bibitem{Beau2}
{A. Beauville}, `Determinantal hypersurfaces’, {\it Michigan Math. J.} 48 (2000), 39--64.

\bibitem{Sigma}
{R. Bielawski}, `Hyperk\"ahler manifolds of curves in twistor spaces',
{\it SIGMA} 10 (2014).

\bibitem{BiCM}
{R. Bielawski}, `Slices to sums of adjoint orbits, the Atiyah-Hitchin manifold, and Hilbert schemes of points', {\it Complex Manifolds} 4 (2017), 16--36.



\bibitem{limit}
{R. Bielawski \and L. Schwachh\"ofer}, `Hypercomplex limits of pluricomplex structures and the Euclidean limit of hyperbolic monopoles', 
{\it Ann. Global Anal. Geom.} 44 (2013), 245--256. 


\bibitem{Bot}
{F. Bottacin}, `Poisson structures on Hilbert schemes of points of a surface and integrable systems', {\it manuscripta math.} 97 (1998), 517--527.


\bibitem{Cav}
{G.R. Cavalcanti}, `Examples and counter-examples of log-symplectic manifolds', {\it J. Topol. } 10 (2017), 1--21.


\bibitem{Rui}
{R.L. Fernandes},  `Completely integrable bi-Hamiltonian systems', {\it J. Dyn. Diff. Equat.} 6 (1994), 53--69.

\bibitem{Goto}
{R. Goto}, `Rozansky-Witten  invariants  of  log  symplectic  manifolds', in {\em Integrable systems, topology, and physics (Tokyo, 2000)}, Contemp. Math., vol. 309, AMS, Providence, RI, 2002, 69--84.

\bibitem{Gua}
{M. Gualtieri \and S. Li}, `Symplectic groupoids of log symplectic manifolds', {\it IMRN} 2014 (11), 3022--3074, 2014.

\bibitem{GMP}
{V. Guillemin, E. Miranda, \and A. R. Pires}, `Symplectic and Poisson geometry on b-manifolds', {\it Adv. Math.} 264 (2014), 864--896.


\bibitem{Hart} {R. Hartshorne}, {\em Deformation Theory}, Springer, New York, 2010.



\bibitem{Hay}
{A. Haydys}, `Nonlinear Dirac operator and quaternionic analysis', {\it  Commun.  Math. Phys.} 281 (2008), 251--261.


\bibitem{Hit}
{N.J. Hitchin}, `Integrable systems in Riemannian geometry’, in: {\em Surveys in differential geom-
etry: integral systems}, 21--81, Int. Press, Boston, 1998.

\bibitem{HS}
{W. Hong \and M. Sti\'enon},
`From hypercomplex to holomorphic symplectic structures', {\it J. Geom. Phys.} 96 (2015), 187--203.



\bibitem{Nicc1} 
{N. Lora Lamia Donin}, `Transverse Hilbert schemes and completely integrable systems', {\it Complex Manifolds} 4 (2017), 263--272.


\bibitem{Nicc2}
{N. Lora Lamia Donin}, `Hyperk\"ahler manifolds of curves and l-hypercomplex structures', Ph.D. Thesis, Leibniz Universit\"at Hannover, 2018.

\bibitem{MM}
{F.Magri \and C.Morosi},  `A   geometrical characterization of integrable Hamiltonian systems through  the theory of Poisson-Nijenhuis manifolds',  {\it Quaderno  S.}  19, Univ.  of  Milan,  1984.

\bibitem{Man}
{C. Manolescu}, `Nilpotent slices, Hilbert schemes, and the Jones polynomial',
{\it Duke Math. J.} 132 (2006), 311--369. 


\bibitem{Sal}
{S.M. Salamon}, `Differential geometry of quaternionic manifolds', {\it Ann. Sci. \'Ec. Norm. Sup\'er.}  Serie 4, Volume 19 (1986), p. 31--55.


\bibitem{Ward}
{R.S. Ward}, `On self-dual gauge fields', {\em Phys. Lett. A} 61 (1977), 81--82.



\end{thebibliography}
\end{document}